\documentclass[10pt]{amsart}
\usepackage{amsmath,dsfont,amsfonts}
\usepackage{amssymb}
\usepackage{amsthm}
\usepackage{verbatim}
\usepackage{color}

\newcommand{\dive}{\text{\normalfont div}}
\numberwithin{equation}{section}

\newtheorem{theo}{Theorem}[section]
\newtheorem{lemma}[theo]{Lemma}
\newtheorem{pr}[theo]{Proposition}

\newtheorem{cor}[theo]{Corollary}

\newtheorem{rem}[theo]{Remark}

\begin{document}

\title[Strong unique continuation for wave equations ]{Strong unique continuation for second order hyperbolic equations with time independent coefficients}

 \author[]{Sergio~Vessella}
 \address{Dipartimento di Matematica e Informatica ``U. Dini'',
Universit\`{a} di Firenze, Italy}
\email{sergio.vessella@unifi.it}


\keywords{Unique Continuation Property, Stability Estimates, Hyperbolic Equations, Inverse Problems.}

\subjclass[2010]{35R25, 35L; Secondary 35B60 ,35R30}

\begin{abstract}
In this paper we prove that if $u$ is a solution to second order hyperbolic equation  $\partial^2_{t}u+a(x)\partial_{t}u-(\dive_x\left(A(x)\nabla_x u\right)+b(x)\cdot\nabla_x u+c(x)u)=0$ and $u$ is flat on a segment $\{x_0\}\times (-T,T)$ then $u$ vanishes in a neighborhood of $\{x_0\}\times (-T,T)$.
\end{abstract}

\maketitle

\section{Introduction} \label{introduction}
In this paper we study strong unique continuation property for the equation
\begin{equation} \label{4i-65-introd}
\partial^2_{t}u+a(x)\partial_{t}u-L(u)=0, \quad \hbox{in } B_{\rho_0}\times(-T,T),
\end{equation}
where $\rho_0, T$ are given positive numbers, $B_{\rho_0}$ is the ball of $\mathbb{R}^{n}$, $n\geq 2$, of radius $\rho_0$ and center at 0, $a\in L^{\infty}(\mathbb{R}^n)$, $L$ is the second order elliptic operator
\begin{equation} \label{L-4i-65-introd}
L(u)=\dive_x\left(A(x)\nabla_x u\right)+b(x)\cdot\nabla_x u+c(x)u,
\end{equation}
$b\in L^{\infty}(\mathbb{R}^n;\mathbb{R}^n)$, $c\in L^{\infty}(\mathbb{R}^n)$ and $A(x)$ is a real-valued symmetric $n\times n$ matrix that satisfies a uniform ellipticity condition and entries of $A(x)$ are functions of Lipschitz class.

We say that the equation \eqref{4i-65-introd} has the strong unique continuation property (SUCP) if there exists a neighborhood $\mathcal{U}$ of $(-T,T)$ such that for every solution, $u$, to equation \eqref{4i-65-introd} we have

\begin{equation}\label{property}
\left\Vert u\right\Vert_{L^2\left(B_{r}\times(-T,T)\right)}=\mathcal{O}(r^N) \mbox{, } \forall N\in\mathbb{N} \mbox{, as } r\rightarrow 0, \quad \Longrightarrow \quad u=0 , \mbox{ in } \mathcal{U}.
\end{equation}

Property \eqref{property} was proved (if the matrix $A$ belongs to $C^2$), , under the additional condition $T=+\infty$ and is $u$ bounded, by Masuda in 1968, \cite{Ma}. Later on, in 1978, Baouendi and Zachmanoglou, \cite{B-Z}, proved the SUCP whenever the coefficients of equation \eqref{4i-65-introd} are analytic functions. In 1999, Lebeau, \cite{Le}, proved the SUCP for solution to \eqref{4i-65-introd} when $a=0$. The proof, rather intricate, of \cite{Le} does not seem to adapt to the case of nonvanishing $a$. We also refer to \cite{Si-Ve}, \cite{Ve2} where  the SUCP at the boundary and the quantitative estimate of unique continuation related to property was proved when $a=0$.

It is worth noting that SUCP and the related quantitative estimates, has been extensively studied and today well understood in the context of second order elliptic and parabolic equation. Among the extensive literature on the subject here we mention, for the elliptic equations, \cite{AKS}, \cite{hormanderpaper1}, \cite{KoTa1}, and, for the parabolic equations, \cite{A-V}, \cite{EsFe}, \cite{KoTa2}. In the context of elliptic and parabolic equations the quantitative estimates of unique continuation appear in the form of three sphere inequalities \cite{La}, doubling inequalities \cite{GaLi}, or two-sphere one cylinder inequality \cite{EsFeVe}. We refer to \cite{A-R-R-V} and \cite{Ve1} for a more extensive literature concerning the elliptic context and the parabolic context respectively.

In the present paper we prove (Theorem \ref{5-115}) a quantitative estimate of unique continuation from which we derive (Corollary \ref{SUCPcor}) property \eqref{property} for equation \eqref{4i-65-introd}. The crucial step of the proof is Proposition \ref{Crucial}, in such a Proposition \ref{Crucial} we exploit in a suitable way the simple and classical idea of converting a hyperbolic equation into an elliptic equation, see for instance \cite[Ch. 6]{G}.

More precisely we define the function

\begin{equation}
\label{def-v-k-introd}
 v_k(x,y)=\int^T_{-T}u(x,t)\varphi_k(t+iy)dt, \mbox{ }\quad \forall k\in\mathbb{N},
\end{equation}
where $\varphi_k(t+iy)$ is a polynomial with the following property:

\medskip

a) $\varphi_k(t+i0)$ is an approximation of Dirac's $\delta$-function,

\smallskip

b) $\varphi_k(\pm T+iy)=\mathcal{O}((C|y|)^k)$ as $k\rightarrow \infty$ and $|y|\leq 1$, where $C$ is a constant.

\medskip

In this way functions  $v_k$ turn out solutions to the elliptic equation
\begin{equation}\label{equ-for-v-introd}
\partial^2_{y}v_k-ia(x)\partial_{y}v_k+L(v_k)=F_k(x,y), \quad \hbox{in } B_{1}\times\mathbb{R},
\end{equation}
where $F_k=\mathcal{O}((C|y|)^k)$, as $k\rightarrow \infty$ and $|y|\leq 1$. This behavior of $F_k$ allows us to handle in a suitable way a Carleman estimate with singular weight for second order elliptic operators, see Section \ref{preliminary} below, in such a way to get $u(x,0)=0$ for $x\in B_{\rho}$, where $\rho\leq \rho_0/C$. Similarly we prove for every $t\in (-T,T)$, $u(\cdot,t)=0$  in $B_{\rho(t)}$, where $\rho(t)=(1-tT^{-1})\rho$. So that we obtain \eqref{property} with $\mathcal{U}=\bigcup_{t\in (-T,T)}(B_{\rho(t)}\times \{t\})$. As a consequence of this result and using the weak unique continuation property proved in \cite{hormanderpaper2}, \cite{Ro-Zu} and  \cite{Ta}, see also \cite{isakovlib2}, and \cite{B-K-L}, for the related quantitative estimates, we have that $u=0$ in the domain of dependence of $\mathcal{U}$.

The quantitative estimate of unique continuation that we prove in Theorem \ref{5-115} can be read, roughly speaking, as a continuous dependence estimate of $u_{_{|\mathcal{U}}}$ from $u_{_{|{B_{r_0}\times(-T,T)}}}$, where $r_0$ is arbitrarily small. The sharp character of such a continuous dependence result is related to the logarithmic character of this estimate, that, at the light of counterexample of John \cite{J}, cannot be improved and to the fact that the quantitative estimate implies the SUCP property. The quantitative estimate of strong unique continuation (at the interior and at the boundary) was a crucial tool, see \cite{Ve3}, to prove sharp stability estimate for inverse problems with unknown boundaries for wave equation $\partial^2_t-\dive_x\left(A(x)\nabla_x u\right)=0$.

Before concluding this Introduction we mention an open question(to the author knowledge). Such an open question concerns the SUCP, \eqref{property}, for the second order hyperbolic equation with coefficients that are analytic in variable $t$ and smooth enough (but not analytic) in variables $x$. This is, for instance, the case of the equation

\begin{equation*} \label{4i-65-introd-t}
\partial^2_{t}u+a(x,t)\partial_{t}u-\Delta_xu=0,
\end{equation*}
where $a(x,t)$ is smooth enough w.r.t $x$ and analytic w.r.t. $t$. Concerning this topic we mention \cite{Lu} in which it is proved that if $u$ satisfies the conditions: (a) $(-T,T)\times \{0\}\cap{\rm supp}u$ is compact and (b $D^{j}u(x,t)=\mathcal{O}\left(e^{-k/|x|}\right)$, $j=1,2$, for every $k$ as $x\rightarrow 0$, $t\in (-T,T)$, then $u$ vanishes in a neighborhood of $(-T,T)$.

\medskip

The plan of the paper is as follows. In Section \ref{MainRe} we state the main result of this paper, in Section \ref{SUCP_Estimates} we prove the main theorem.

\section{The main results}\label{MainRe}
\subsection{Notation and Definition} \label{sec:notation}
Let $n\in\mathbb{N}$, $n\geq 2$. For any $x\in\mathbb{R}^n$, we will denote $x=(x',x_n)$, where $x'=(x_1,\ldots,x_{n-1})\in\mathbb{R}^{n-1}$, $x_n\in\mathbb{R}$ and $|x|=\left(\sum_{j=1}^nx_j^2\right)^{1/2}$.
Given $r>0$, we will denote by $B_r$, $B'_r$ $\widetilde{B}_r$ the ball of $\mathbb{R}^{n}$, $\mathbb{R}^{n-1}$ and $\mathbb{R}^{n+1}$ of radius $r$ centered at 0.
For any open set $\Omega\subset\mathbb{R}^n$ and any function (smooth enough) $u$  we denote by $\nabla_x u=(\partial_{x_1}u,\cdots, \partial_{x_n}u)$ the gradient of $u$. Also, for the gradient of $u$ we use the notation $D_xu$. If $j=0,1,2$ we denote by $D^j_x u$ the set of the derivatives of $u$ of order $j$, so $D^0_x u=u$, $D^1_x u=\nabla_x u$ and $D^2_xu$ is the hessian matrix $\{\partial_{x_ix_j}u\}_{i,j=1}^n$. Similar notation are used whenever other variables occur and $\Omega$ is an open subset of $\mathbb{R}^{n-1}$ or a subset $\mathbb{R}^{n+1}$. By $H^{\ell}(\Omega)$, $\ell=0,1,2$, we denote the usual Sobolev spaces of order $\ell$ (in particular, $H^0(\Omega)=L^2(\Omega)$), with the standard norm
\begin{equation*}
 \left\Vert v(x)\right\Vert_{H^{\ell}(\Omega)}=\left(\sum_{0\leq j\leq \ell}\int_{\Omega}\left\vert D^jv(x)\right\vert^2dx\right)^{1/2}.
\end{equation*}

For any interval $J\subset \mathbb{R}$ and $\Omega$ as above we denote
\[\mathcal{W}\left(J;\Omega\right)=\left\{u\in C^0\left(J;H^2\left(\Omega\right)\right): \partial_t^\ell u\in C^0\left(J;H^{2-\ell}\left(\Omega\right)\right), \ell=1,2\right\}.\]

We shall use the letters $c, C,C_0,C_1,\cdots$ to denote constants. The value of the constants may change from line to line, but we shall specified their dependence everywhere they appear. Generally we will omit the dependence of various constants by $n$.

\subsection{Statements of the main results}\label{QEsucp}

Let $\rho_0>0$, $T$, $\lambda\in(0,1]$, $\Lambda>0$ and $\Lambda_1>0$  be given number. Let $A(x)=\left\{a^{ij}(x)\right\}^n_{i,j=1}$ be a real-valued symmetric $n\times n$ matrix whose entries are measurable functions and they satisfy the following conditions
\begin{subequations}
\label{1-65}
\begin{equation}
\label{1-65a}
\lambda\left\vert\xi\right\vert^2\leq A(x)\xi\cdot\xi\leq\lambda^{-1}\left\vert\xi\right\vert^2, \quad \hbox{for every } x, \xi\in\mathbb{R}^n,
\end{equation}
\begin{equation}
\label{2-65}
\left\vert A(x_{\ast})-A(x)\right\vert\leq\frac{\Lambda}{\rho_0} \left\vert x-y \right\vert, \quad \hbox{for every } x_{\ast}, x\in\mathbb{R}^n.
\end{equation}
\end{subequations}

Let $b\in L^{\infty}(\mathbb{R}^n; \mathbb{R}^n)$ and $a, c\in L^{\infty}(\mathbb{R}^n)$ satisfy

\begin{equation} \label{bound-b}
 T\left\vert a(x)\right\vert+T^2\rho_0^{-1} \left\vert b(x)\right\vert+T^{2}\left\vert c(x)\right\vert\leq\Lambda_1, \quad \hbox{for almost every } x\in\mathbb{R}^n,
\end{equation}

Let
\begin{equation} \label{L-4i-65}
L(u)=\dive_x\left(A(x)\nabla_x u\right)+b(x)\cdot\nabla_x u+c(x)u.
\end{equation}

Let $u\in\mathcal{W}\left([-T,T];B_{\rho_0}\right)$ be a solution to

\begin{equation} \label{4i-65}
\partial^2_{t}u+a(x)\partial_{t}u-L(u)=0, \quad \hbox{a.e. } \hbox{in } B_{\rho_0}\times(-T,T).
\end{equation}
Let $\varepsilon$ and $H$ be given positive numbers and let $r_0\in (0,\rho_0]$. We assume

\begin{equation} \label{errore-verticale}
\rho_0^{-n}T^{-1}\int_{-T}^{T}\int_{B_{r_0}}\left\vert u(x,t)\right\vert^2dxdt\leq\varepsilon^2
\end{equation}
and

\begin{equation} \label{limitaz-a-priori}
\max_{t\in [-T,T]}\left(\rho_0^{-n}\int_{B_{\rho_0}}\left\vert u(x,t)\right\vert^2 dx+\rho_0^{-n+1}\int_{B_{\rho_0}}\left\vert \partial_tu(x,t)\right\vert^2dx\right)\leq H^2.
\end{equation}

\begin{theo}\label{5-115}
Let $u\in\mathcal{W}\left([-T,T];B_{\rho_0}\right)$ be a weak solution to \eqref{4i-65} and let \eqref{1-65}, \eqref{bound-b}, \eqref{errore-verticale} and \eqref{limitaz-a-priori} be satisfied. For every $\alpha\in (0,1/2)$ there exist constants $s_0\in (0,1)$ and $C\geq 1$ depending on $\lambda$, $\Lambda$, $\Lambda_1$, $\alpha$ and $T\rho_0^{-1}$ only such that for every $t_0\in (-T,T)$ and every $0<r_0\leq \rho\leq s_0\rho_0$ the following inequality holds true
\begin{equation}
\label{SUCP-final}
\rho_0^{-n}\int_{B_{(1-|t_0|T^{-1})\rho}}\left\vert u(x,t_0) \right\vert^2dx\leq \frac{C(\rho_0\rho^{-1})^C(H+e\varepsilon)^2}{\left(\theta\log \left( \frac{H+e\varepsilon}{\varepsilon}\right) \right)^{\alpha}},
\end{equation}
where
\begin{equation}
\label{theta-0}
\theta_0=\frac{\log (\rho_0/C\rho)}{\log (\rho_0/r_0)}.
\end{equation}
\end{theo}
The proof of Theorem \ref{5-115} is given in Section \ref{SUCP_Estimates}.

\bigskip

The proof of the following Corollary is standard (see, for instance, \cite[Remark 2.2]{Ve2}), but we give it for the reader convenience.
\begin{cor}[\textbf{Strong Unique Continuation Property}]\label{SUCPcor}
Let $u\in\mathcal{W}\left([-T,T];B_{\rho_0}\right)$ be a weak solution to \eqref{4i-65}. Assume that \eqref{1-65} and \eqref{bound-b} be satisfied. We have that, if
$$\left(\rho_0^{-n}T^{-1}\int^T_{-T}\int_{B_{r_0}}\left\vert u(x,t)\right\vert^2dxdt\right)^{1/2}=O(r_0^N)\mbox{, } \forall N\in\mathbb{N} \mbox{, as } r_0\rightarrow 0,$$
then
\begin{equation}
\label{ucp-rem}
u(\cdot,t)=0 \mbox{, for } |x|+\frac{\rho_0|t|}{T}\leq s_0\rho_0.
\end{equation}

\end{cor}
\begin{proof}
We consider the case $t=0$, similarly we could proceed for $t\neq0$. If $\left\Vert u(\cdot,0) \right\Vert_{L^2\left(B_{s_0\rho_0}\right)}=0$ there is nothing to proof, otherwise, if
\begin{equation}
\label{ucp-cor-1}
\left\Vert u(\cdot,0) \right\Vert_{L^2\left(B_{s_0\rho_0}\right)}>0
\end{equation}
we argue by contradiction. By \eqref{ucp-cor-1} it is not restrictive to assume that
\begin{equation} \label{limitaz-a-priori-cor}
\max_{t\in [-T,T]}\left(\rho_0^{-n}\int_{B_{\rho_0}}\left\vert u(x,t)\right\vert^2 dx+\rho_0^{-n+1}\int_{B_{\rho_0}}\left\vert \partial_tu(x,t)\right\vert^2dx\right)=1.
\end{equation}
Now we apply inequality \eqref{SUCP-final} with $\varepsilon=C_Nr_0^N$, $H=1$ and passing to the limit as $r_0\rightarrow 0$ we derive
\begin{equation}
\label{ucp-cor-2}
\left\Vert u(\cdot,0) \right\Vert_{L^2\left(B_{s_0\rho_0}\right)}\leq Cs_0^{-C}N^{-\alpha/2} \mbox{, } \forall N\in\mathbb{N},
\end{equation}
by passing again to the limit as $N\rightarrow 0$, by \eqref{ucp-cor-2}, we obtain $\left\Vert u(\cdot,0) \right\Vert_{L^2\left(B_{s_0\rho_0}\right)}=0$ that contradicts \eqref{ucp-cor-1}.
\end{proof}

\subsection{Auxiliary result: Carleman estimate with singular weight} \label{preliminary} In order to prove Theorem \ref{5-115} we need a Carleman estimate proved by several authors, here we recall \cite{AKS}, \cite{hormanderpaper1}. In order to control the dependence of the various constants, we use here a version of such a Carleman estimate proved, in the context of parabolic operator, in  \cite{EsVe}, see also \cite[Section 8]{BK}.

First we introduce some notation.
Let $\mathcal{P}$ be the elliptic operator
\begin{equation}
\label{oper}
\mathcal{P}(w):=\partial^2_{y}w+L(w)-ia(x)\partial_{y}w.
\end{equation}
Denote
\begin{equation}
\label{6-96}
\varrho(x,y)=\left(A^{-1}(0)x\cdot x+y^2\right)^{1/2},
\end{equation}

\begin{equation}
\label{palla n+1}
\widetilde{B}^{\varrho}_{r}= \left\{(x,y)\in \mathbb{R}^{n+1}: \varrho(x,y)< r\right\}, \quad \hbox{  } r>0.
\end{equation}

Notice that
\begin{equation}
\label{5-119}
\widetilde{B}^{\varrho}_{\sqrt{\lambda} r}\subset \widetilde{B}_r \subset \widetilde{B}^{\varrho}_{r/\sqrt{\lambda}}, \quad \hbox{ }\forall r>0.
\end{equation}

\begin{theo}\label{Carleman}
Let $\mathcal{P}$ be the operator \eqref{oper} and assume that \eqref{1-65} is satisfied. There exists constants $C_{\ast}>1$ depending on $\lambda$, and $\Lambda$ only and  $\tau_0>1$ depending on $\lambda$, $\Lambda$ and $\Lambda_1$ only such that, denoting

\begin{subequations}
\label{4-5-6-96}
\begin{equation}
\label{4-96}
\Psi(r)=r\exp \left(\int^r_0\frac{e^{-C_{\ast}\eta}-1}{\eta}d\eta\right),
\end{equation}
\begin{equation}
\label{5-96}
\psi(x,y)=\Psi\left(\varrho(x,y)/2\sqrt{\lambda}\right),
\end{equation}
\end{subequations}
for every $\tau\geq \tau_0$ and $w\in C^{\infty}_0\left(\widetilde{B}^{\varrho}_{2\sqrt{\lambda}/C_{\ast}}\setminus\{0\}\right)$ we have

\begin{gather}
\label{6-118}
\int_{\mathbb{R}^{n+1}}\left(\tau\psi^{1-2\tau}\left\vert\nabla_{x,y}w\right\vert^2+
\tau^3\psi^{-1-2\tau} \left\vert w\right\vert^2\right) dxdy
\leq C_{\ast}\int_{{\mathbb{R}^{n+1}}}\psi^{2-2\tau} \left\vert \mathcal{P}(w)\right\vert^2dxdy.
\end{gather}

\end{theo}

\begin{rem}
We emphasize that $$\Psi(r)\simeq r, \quad \hbox{ as } r\rightarrow 0.$$ Moreover $\Psi$ is an increasing and concave function and there exists $C>1$ depending on $\lambda$, and $\Lambda$ such that
\begin{equation}
\label{Psi}
C^{-1}r\leq\Psi\left(r\right)\leq r, \quad \hbox{ } \forall r\in (0,1].
\end{equation}

\end{rem}

\section{Proof of Theorem \ref{5-115}} \label{SUCP_Estimates}
The primary step to achieve Theorem \ref{5-115} consists in proving the following

\begin{pr}\label{Crucial}
Let us assume $\rho_0=1$ and $T=1$. Let $u\in\mathcal{W}\left([-1,1];B_1\right)$ be a weak solution to \eqref{4i-65} and let \eqref{1-65}, \eqref{bound-b}, \eqref{errore-verticale} and \eqref{limitaz-a-priori} be satisfied. For every $\alpha\in (0,1/2)$ there exist constants $s_0\in (0,1)$ and $C\geq 1$ depending on $\lambda$, $\Lambda$, $\Lambda_1$ and $\alpha$ only such that for every $0<r_0\leq s\leq s_0$ the following inequality holds true

\begin{equation}
\label{SUCP}
\int_{B_{s}}\left\vert u(x,0) \right\vert^2dx\leq \frac{Cs^{-C}(H+e\varepsilon)^2}{\left(\theta\log \left( \frac{H+e\varepsilon}{\varepsilon}\right) \right)^{\alpha}},
\end{equation}
where
\begin{equation}
\label{theta}
\theta=\frac{\log (1/Cs)}{\log (1/r_0)}.
\end{equation}
\end{pr}

\bigskip

In order to prove Proposition \ref{Crucial} we define
\begin{equation}
\label{def-v-k}
 v_k(x,y)=\int^1_{-1}u(x,t)\varphi_k(t+iy)dt, \mbox{ }\quad \forall k\in\mathbb{N},
\end{equation}
where
\begin{equation}
\label{def-phi}
\varphi_k(z)=\mu_k\left(1-z^2\right)^k, \mbox{ }\quad z=t+iy\in\mathbb{C},
\end{equation}
and
\begin{equation}
\label{def-mu}
\mu_k=\left(\int^1_{-1}\left(1-t^2\right)^kdt\right)^{-1},
\end{equation}
so that we have
\begin{equation}
\label{delta-approx}
\int^1_{-1}\varphi_k(t)dt=1, \quad\mbox{ } \forall k\in\mathbb{N}.
\end{equation}
It is easy to check that
\begin{equation}
\label{asimptotic-mu}
\mu_k\simeq \sqrt{\frac{k}{\pi}}, \quad\mbox{ as } k\rightarrow\infty.
\end{equation}

\bigskip

We need some simple lemmas to state the properties of functions $v_k$.

\begin{lemma}\label{Lemma1}
We have
\begin{equation}\label{stima-1-31s}
\left\Vert v_k(\cdot,0)-u(\cdot,0)\right\Vert_{L^2(B_1)}\leq c\frac{\log k}{ \sqrt{k}}, \mbox{ }\quad \forall k\in\mathbb{N},
\end{equation}
where $c$ depends on $n$ only.
\end{lemma}

\begin{proof}
By \eqref{def-v-k} and \eqref{delta-approx} we have

\begin{equation*}
 v_k(x,0)-u(x,0)=\int^1_{-1}\left(u(x,t)-u(x,0)\right)\varphi_k(t)dt, \quad\mbox{ } \forall x\in B_1,
 \end{equation*}
hence, by Schwarz inequality and integrating over $B_1$ we have

\begin{gather}\label{pag-29s-1}
 \int_{B_1}\left\vert v_k(x,0)-u(x,0)\right\vert^2dx\leq \int_{B_1}dx\int^1_{-1}\left\vert u(x,t)-u(x,0)\right\vert^2\varphi_k(t)dt=\\ \nonumber
 =\int_{[-\gamma,\gamma]}\varphi_k(t)dt\int_{B_1}\left\vert u(x,t)-u(x,0)\right\vert^2dx+\int_{[-1,1]\setminus [-\gamma,\gamma]}\varphi_k(t)dt\int_{B_1}\left\vert u(x,t)-u(x,0)\right\vert^2dx,
\end{gather}
where $\gamma\in (0,1)$ is a number that we will choose. Now we have
\begin{equation*}
\varphi_k(t)\leq \mu_k(1-\gamma^2)^{k/2}, \quad\mbox{ } \forall t\in[-1,1]\setminus [-\gamma,\gamma]
\end{equation*}
and, by \eqref{limitaz-a-priori},

\begin{equation*}
\int_{B_1}\left\vert u(x,t)-u(x,0)\right\vert^2dx\leq t^2H^2 .
\end{equation*}
Hence, by \eqref{limitaz-a-priori}, \eqref{asimptotic-mu} and \eqref{pag-29s-1}, we have

\begin{gather}\label{pag-30s-1}
 \left\Vert v_k(\cdot,0)-u(\cdot,0)\right\Vert_{L^2(B_1)}\leq c H \left(\gamma+k^{1/4}(1-\gamma^2)^{k/2}\right), \quad\mbox{for every } \gamma\in (0,1),
\end{gather}
where $c$ depends on $n$ only. Now, we choose $\gamma=k^{-1/2}\log k$ and we get \eqref{stima-1-31s}.
\end{proof}

\begin{lemma}\label{Lemma2}
Let $u$ be a solution to \eqref{4i-65} and let \eqref{1-65} and \eqref{bound-b} be satisfied, then $v_k\in H^2\left(B_1\times(-1,1)\right)$ is  a solution to the equation

\begin{equation}\label{equ-for-v}
\partial^2_{y}v_k-ia(x)\partial_{y}v_k+L(v_k)=F_k(x,y), \quad \hbox{in } B_{1}\times\mathbb{R},
\end{equation}
where $F_k\in L^{\infty}(-1,1;L^2(B_1))$ and it satisfies
\begin{equation}
\label{bound-for-F-k}
\left\Vert F_k(\cdot,y)\right\Vert_{L^2(B_1)}\leq CH k\mu_k |\sqrt{5}y|^{k-1}, \quad\mbox{ } \forall y\in[-1,1],
\end{equation}
$C$ depending on $\Lambda_1$ only.

In addition, $v_k$ satisfies the following properties
\begin{equation}\label{bound-for-v-k-c}
\sup_{y\in[-1,1]}\left\Vert v_k(\cdot,y)\right\Vert_{L^2(B_1)}\leq 2^k\mu_k H,
\end{equation}
\begin{equation}\label{bound-for-v-k-caccioppoli-2}
\int_{\widetilde{B}_{r_0/2}} \left(\left\vert v_k\right\vert^2+r_0^2\left\vert \nabla_{x,y} v_k\right\vert^2\right) dxdy\leq C\left(r_04^kk\varepsilon^2+H^2k^3\left(\sqrt{5}r_0\right)^{2(k+2)}\right),
\end{equation}
where $C$ depends on $\lambda$ and $\Lambda_1$ only.
\end{lemma}

\begin{proof}
The fact that $v_k$ belongs to $H^2\left(B_1\times(-1,1)\right)$ is an immediate consequence of differentiation under the integral sign. Actually we have

\begin{gather}\label{derivate di u-risp x}
\partial_y^m D^j_xv_k(x,y)=\int^1_{-1}D^j_x u(x,t)\partial_y^m\left(\varphi_k(t+iy)\right)dt, \quad\mbox{ for } j,m=0,1,2,
\end{gather}
hence by Schwarz inequality and taking into account that $u\in\mathcal{W}\left([-1,1];B_1\right)$ we have $v_k\in H^2\left(B_1\times(-1,1)\right)$.

Now we prove \eqref{equ-for-v}.

By integration by parts and taking into account that

\begin{equation*}
\partial_t \varphi_k(t+iy)=\frac{1}{i}\partial_y \varphi_k(t+iy),
\end{equation*}
we have
\begin{gather*}
\int^1_{-1}\partial_t u(x,t)\varphi_k(t+iy)dt=\nonumber\\
=\left.(u(x,t)\varphi_k(t+iy))\right \vert_{t=-1}^{t=1}-\int^1_{-1}u(x,t)\partial_t \varphi_k(t+iy)dt=\nonumber\\
=\left.(u(x,t)\varphi_k(t+iy))\right \vert_{t=-1}^{t=1}-\frac{1}{i}\int^1_{-1}u(x,t)\partial_y \varphi_k(t+iy)dt=\nonumber\\
=\left.(u(x,t)\varphi_k(t+iy))\right \vert_{t=-1}^{t=1}+i\partial_yv_k(x,y).
\end{gather*}
Hence we have
\begin{gather}\label{derivata-prima-u}
-i\partial_yv_k(x,y)=-\int^1_{-1}\partial_t u(x,t)\varphi_k(t+iy)dt+\left.(u(x,t)\varphi_k(t+iy))\right\vert_{t=-1}^{t=1}.
\end{gather}
Similarly we have

\begin{gather}
\partial^2_yv_k(x,y)=-\int^1_{-1}\partial^2_tu(x,t)\varphi_k(t+iy)dt+\nonumber\\
+\left. (\partial_tu(x,t)\varphi_k(t+iy))\right\vert_{t=-1}^{t=1}-\left. (u(x,t)\varphi'_k(t+iy))\right\vert_{t=-1}^{t=1}.
\label{s}
\end{gather}

\medskip

Now, by \eqref{4i-65}, \eqref{derivate di u-risp x}, \eqref{derivata-prima-u} and \eqref{s} we have
\begin{gather}
\partial^2_{y}v_k-ia(x)\partial_{y}v_k+L(v_k)=\nonumber\\
=-\int^1_{-1}\left\{\partial^2_tu(x,t)+a(x)\partial_tu(x,t)-L(u)(x,t)\right\}\varphi_k(t+iy)dt+F_k(x,y)=F_k(x,y),
\label{2-4s}
\end{gather}
where

\begin{gather}\label{def-F-k}
F_k(x,y)=\varphi_k(1+iy)\left(\vphantom{1^2}a(x)u(x,1)+\partial_t u(x,1)\right)-\varphi'_k(1+iy)u(x,1)-\nonumber\\
-\left[\vphantom{1^2}\varphi_k(-1+iy)\left(a(x)u(x,-1)+\partial_t u(x,-1)\right)-\varphi'_k(-1+iy)u(x,-1)\right]
\end{gather}
and \eqref{equ-for-v} is proved.

Now we prove \eqref{bound-for-F-k}.

It is easy to check that, for every $k\in\mathbb{N}$ we have
\begin{subequations}\label{bound-for-varphi-k}
\begin{equation}\label{bound-for-varphi-k-a}
\left\vert \varphi_k(\pm 1+iy)\right\vert=\mu_k\left(4y^2+y^4\right)^{k/2}, \quad\mbox{ } \forall y\in \mathbb{R},
\end{equation}
\begin{equation}\label{bound-for-varphi-k-b}
\left\vert \varphi'_k(\pm 1+iy)\right\vert=2k\mu_k\left(1+y^2\right)^{1/2}\left(4y^2+y^4\right)^{(k-1)/2}, \quad\mbox{ } \forall y\in \mathbb{R}.
\end{equation}
\end{subequations}
In addition, since
\begin{equation*}
\left\vert \varphi_k(t+iy)\right\vert=\mu_k\left[t^4-2t^2(1-y^2)+(1+y^2)^2\right]^{\frac{k}{2}},
\end{equation*}
we have
\begin{equation}\label{bound-for-varphi-k-c}
\left\vert \varphi_k(t+iy)\right\vert\leq2^k\mu_k, \quad\mbox{ } \forall (t,y)\in[-1,1]\times[-1,1].
\end{equation}

By \eqref{bound-b}, \eqref{limitaz-a-priori}, \eqref{bound-for-varphi-k-a} and \eqref{bound-for-varphi-k-b} we have \eqref{bound-for-F-k}.

\medskip

By Schwarz inequality and \eqref{bound-for-varphi-k-c} we have, for any $R\in(0,1]$,

\begin{equation}\label{bound-for-v-k-R}
\sup_{y\in[-1,1]}\left\Vert v_k(\cdot,y)\right\Vert_{L^2(B_{R})}\leq 2^k\mu_k\left(\int_{-1}^{1}\int_{B_{R}}u^2(x,t)dxdt\right)^{1/2},
\end{equation}
hence, for $R=1$, taking into account \eqref{limitaz-a-priori}, we obtain \eqref{bound-for-v-k-c}.

Finally, let us prove \eqref{bound-for-v-k-caccioppoli-2}. To this purpose we firstly observe that applying \eqref{bound-for-v-k-R} for $R=r_0$ and taking into account \eqref{errore-verticale} we have
\begin{equation}\label{bound-for-v-k-Nuovo}
\sup_{y\in[-1,1]}\left\Vert v_k(\cdot,y)\right\Vert_{L^2(B_{r_0})}\leq C 2^{k}\mu_k\varepsilon.
\end{equation}
Afterwards, since $v_k$ is solution to elliptic equation \eqref{equ-for-v}, the following Caccioppoli's inequality, \cite{caccioppoli}, \cite{Gi}, holds
\begin{equation}\label{bound-for-v-k-caccioppoli-1}
\int_{\widetilde{B}_{r_0/2}} \left\vert \nabla_{x,y} v_k\right\vert^2dxdy\leq Cr_0^{-2}\int_{\widetilde{B}_{r_0}} \left(\left\vert v_k\right\vert^2+r_0^{4}\left\vert F_k\right\vert^2\right) dxdy,
\end{equation}
where $C$ depends on $\lambda$ only. Finally, by \eqref{asimptotic-mu}, \eqref{bound-for-F-k}, \eqref{bound-for-v-k-Nuovo} and \eqref{bound-for-v-k-caccioppoli-1} we get \eqref{bound-for-v-k-caccioppoli-2}.
\end{proof}

\bigskip

\textit{Proof of Proposition \ref{Crucial}}

Set $$r_1=\frac{\sqrt{\lambda} r_0}{8}.$$ By \eqref{bound-for-v-k-caccioppoli-2} we have

\begin{equation}
\label{2-120}
\int_{\widetilde{B}^{\varrho}_{4r_1}} \left(\left\vert v_k\right\vert^2+r_1^2\left\vert \nabla_{x,y} v_k\right\vert^2\right) dxdy\leq C r_1\sigma_k(\varepsilon, r_1),
\end{equation}
where $C$ depends on $\lambda$ and $\Lambda_1$ only and
\begin{equation}
\label{3-120}
 \sigma_k(\varepsilon, r_1)=4^kk\varepsilon^2+H^2k^3\left(C_1r_1\right)^{2k+2},
\end{equation}
where $C_1=2\sqrt{5}\lambda^{-1/2}$.

Now we apply Theorem \ref{Carleman}.

Denote
\[\psi_0(r):=\Psi(r/2\sqrt{\lambda})\quad\hbox{, for every } r>0\]
and $$R=\frac{\sqrt{\lambda}}{2C_{\ast}}.$$

Let us define
\begin{equation*}
\zeta(x,y)=h\left(\psi(x,y)\right).
\end{equation*}
where $h$ belongs to $C^2_0\left(0, \psi_0\left(2R\right)\right)$ and satisfies

\begin{subequations}
\begin{equation*}
0\leq h\leq 1,
\end{equation*}
\begin{equation*}
h(r)=1 ,\quad\hbox{ } \forall r\in\left[\psi_0\left(2r_1\right), \psi_0\left(R\right)\right],
\end{equation*}
\begin{equation*}
h(r)=0, \quad\hbox{ } \forall r\in\left[0,\psi_0\left(r_1\right)\right]\cup \left[\psi_0\left(3R/2\right), \psi_0\left(2R\right)\right],
\end{equation*}
\begin{equation*}
r_1\left\vert h'(r)\right\vert+r_1^2\left\vert h''(r)\right\vert\leq c, \quad\hbox{ } \forall r\in\left[\psi_0\left(r_1\right), \psi_0\left(2r_1\right)\right],
\end{equation*}
\begin{equation*}
\left\vert h'(r)\right\vert +\left\vert h''(r)\right\vert\leq c, \quad\hbox{ } \forall r\in\left[\psi_0\left(R\right), \psi_0\left(3R/2\right)\right],
\end{equation*}
\end{subequations}
where $c$ depends on $\lambda$ and $\Lambda$ only. Notice that if $2r_1\leq \varrho(x,y)\leq R$ then $\zeta(x,y)=1$ and if $\varrho(x,y)\geq 2R$ or $\varrho(x,y)\leq r_1$ then $\zeta(x,y)=0$.

By density, we can apply \eqref{6-118} to the function $w=\zeta v_k$ and we have, for every $\tau\geq \tau_0$,

\begin{equation}
\label{1-121}
\int_{\widetilde{B}^{\varrho}_{2R}}\left(\tau\psi^{1-2\tau}\left\vert\nabla_{x,y}\left(\zeta v_k\right)\right\vert^2+
\tau^3\psi^{-1-2\tau}\left\vert\zeta v_k\right\vert^2\right)dxdy
\leq  C\left(I_1+I_2+I_2\right),
\end{equation}
where $C$ depends on $\lambda$, $\Lambda$ and $\Lambda_1$ only and
\begin{subequations}\label{def-di-I}
\begin{equation}\label{def-di-I-a}
I_1=\int_{\widetilde{B}^{\varrho}_{2R}}\psi^{2-2\tau} \left\vert F_k\right\vert^2 \zeta^2dxdy,
\end{equation}
\begin{equation}\label{def-di-I-b}
I_2=\int_{\widetilde{B}^{\varrho}_{2R}}\psi^{2-2\tau} \left\vert \mathcal{P}(\zeta)\right\vert^2 \left\vert v_k \right\vert^2dxdy,
\end{equation}
\begin{equation}\label{def-di-I-c}
I_3=\int_{\widetilde{B}^{\varrho}_{2R}}\psi^{2-2\tau} \left\vert\nabla_{x,y} v_k\right\vert^2\left\vert\nabla_{x,y}\zeta\right\vert^2dxdy.
\end{equation}
\end{subequations}

\bigskip

\textbf{Estimate of $I_1$}.

By \eqref{Psi} we have
\begin{equation}
\label{2-122}
C_2^{-1}\leq\left(|x|^2+y^2\right)^{-1/2}\psi(x,y)\leq C_2, \quad \mbox{ } \forall (x,y)\in\widetilde{B}_1,
\end{equation}
where $C_2>2$ depends on $\lambda$ and $\Lambda$ only.

By \eqref{bound-for-F-k}, \eqref{def-di-I-a} and \eqref{2-122} we have

\begin{gather}
\label{I-1}
I_1=\int_{\widetilde{B}^{\varrho}_{2R}}\psi^{2-2\tau} \left\vert F_k\right\vert^2 \zeta^2dxdy\leq C_2^{2(\tau-1)}\int_{\widetilde{B}^{\varrho}_{2R}}\left(|x|^2+y^2\right)^{1-\tau}
\left\vert F_k\right\vert^2 dxdy\leq\\ \nonumber
\leq CH^2k^35^kC_2^{2(\tau-1)}\int_{\widetilde{B}^{\varrho}_{2R}}\left(|x|^2+y^2\right)^{k-\tau}dxdy,
\end{gather}
where $C$ depends on $\lambda$ and $\Lambda$ only.

Now let $k$ and $\tau$ satisfy

\begin{equation}
\label{1-123}
k\geq \tau\geq \tau_0.
\end{equation}
By \eqref{I-1} and \eqref{1-123} we have

\begin{equation}
\label{2-123}
I_1\leq C H^2 k^3 5^kC_2^{2(k-1)}.
\end{equation}

\bigskip

\textbf{Estimate of $I_2$}.

By \eqref{bound-for-v-k-c}, \eqref{2-120} and \eqref{def-di-I-b}  we have
\begin{gather*}
I_2 \leq Cr_1^{-4}\int_{\widetilde{B}^{\varrho}_{2r_1}\setminus\widetilde{B}^{\varrho}_{r_1}}\psi^{2-2\tau} \left\vert v_k\right\vert^2 dxdy
+C\int_{\widetilde{B}^{\varrho}_{3R/2}\setminus\widetilde{B}^{\varrho}_{R}}\psi^{2-2\tau}\left\vert v_k\right\vert^2 dxdy \leq\\
\leq C\left(r_1^{-3}\psi^{2-2\tau}_0(r_1)\sigma_k(\varepsilon, r_1)+H^2k4^k\psi^{2-2\tau}_0(R)\right),
\end{gather*}
hence, by \eqref{2-122} we have

\begin{equation}\label{1-125}
I_2 \leq C\left(\psi^{-1-2\tau}_0(r_1)\sigma_k(\varepsilon, r_1)+H^2k4^k\psi^{1-2\tau}_0(R)\right),
\end{equation}
where $C$ depends on $\lambda$ and $\Lambda$ only.

\bigskip

\textbf{Estimate of $I_3$.}

By \eqref{def-di-I-c}  we have

\begin{gather}\label{2-125}
I_3 \leq Cr_1^{-2}\psi^{2-2\tau}_0(r_1)\int_{\widetilde{B}^{\varrho}_{2r_1}\setminus\widetilde{B}^{\varrho}_{r_1}} \left\vert\nabla_{x,y} v_k\right\vert^2 dxdy+\\ \nonumber+C\psi^{2-2\tau}_0(R)\int_{\widetilde{B}^{\varrho}_{3R/2}\setminus\widetilde{B}^{\varrho}_{R}} \left\vert\nabla_{x,y} v_k\right\vert^2 dxdy.
\end{gather}
Now in order to estimate from above the righthand side of \eqref{2-125} we use the Caccioppoli inequality, \eqref{bound-for-F-k}, \eqref{bound-for-v-k-c} and \eqref{2-120} and we get
\begin{gather}
\label{1-127}
I_3 \leq Cr_1^{-2}\psi^{2-2\tau}_0(r_1)\left(r_1^{-2}\int_{\widetilde{B}^{\varrho}_{4r_1}\setminus\widetilde{B}^{\varrho}_{r_1/2}} \left\vert v_k\right\vert^2 dxdy+r_1^2\int_{\widetilde{B}^{\varrho}_{4r_1}\setminus\widetilde{B}^{\varrho}_{r_1/2}} \left\vert F_k\right\vert^2 dxdy\right)+\\ \nonumber+C\psi^{2-2\tau}_0(R)\left(R^{-2}\int_{\widetilde{B}^{\varrho}_{2R}\setminus\widetilde{B}^{\varrho}_{R/2}} \left\vert v_k\right\vert^2 dxdy +R^2\int_{\widetilde{B}^{\varrho}_{2R}\setminus\widetilde{B}^{\varrho}_{R/2}} \left\vert F_k\right\vert^2 dxdy \right)\leq\\ \nonumber \leq C \sigma_k(\varepsilon,r_1)
\psi^{-1-2\tau}_0(r_1)+CH^25^{k}k^3\psi^{1-2\tau}_0(R):=\widetilde{I}_3,
\end{gather}
where $C$ depends on $\lambda$, $\Lambda$ and $\Lambda_1$ only.

\bigskip

Let $r_1\leq \frac{R}{2}$ and let $s$ be such that $\frac{2r_1}{\sqrt{\lambda}}\leq
s\leq\frac{R}{\sqrt{\lambda}}$. Denote $$\widetilde{s}=\sqrt{\lambda}s.$$ By estimating from below trivially the left hand side of \eqref{1-121} and taking into account \eqref{1-127} we get

\begin{equation}
\label{2-127}
\psi^{-1-2\tau}_0(\widetilde{s})\int_{\widetilde{B}^{\varrho}_{\widetilde{s}}\setminus\widetilde{B}^{\varrho}_{2r_1}}\left\vert v_k\right\vert^2+\psi^{1-2\tau}_0(\widetilde{s})\int_{\widetilde{B}^{\varrho}_{\widetilde{s}}\setminus\widetilde{B}^{\varrho}_{2r_1}}\left\vert\nabla_{x,y} v_k\right\vert^2
\leq C\left(I_1+I_2+\widetilde{I}_3\right),
\end{equation}
where $C$ depends on $\lambda$, $\Lambda$ and $\Lambda_1$ only.

Now, by \eqref{5-119}, \eqref{2-120} and into account that $\psi_0(\widetilde{s})\geq\psi_0(r_1)$ we have
\begin{gather}
\psi^{-1-2\tau}_0(\widetilde{s})\int_{\widetilde{B}^{\varrho}_{2r_1}} \left\vert v_k\right\vert^2+\psi^{1-2\tau}_0(\widetilde{s})\int_{\widetilde{B}^{\varrho}_{2r_1}}\left\vert \nabla_{x,y} v_k\right\vert^2 dxdy\leq \nonumber\\
\leq C\psi^{-1-2\tau}_0(\widetilde{s})\int_{\widetilde{B}^{\varrho}_{2r_1}} \left(\left\vert v_k\right\vert^2+r_1^2\left\vert \nabla_{x,y} v_k\right\vert^2 \right)dxdy\leq C r_1\sigma_k(\varepsilon, r_1)\psi^{-1-2\tau}_0(r_1).
\label{aggiustamento}
\end{gather}

Now let us add at both the side of \eqref{2-127} the quantity

$$\psi^{-1-2\tau}_0(\widetilde{s})\int_{\widetilde{B}^{\varrho}_{2r_1}} \left\vert v_k\right\vert^2+\psi^{1-2\tau}_0(\widetilde{s})\int_{\widetilde{B}^{\varrho}_{2r_1}}\left\vert \nabla_{x,y} v_k\right\vert^2 dxdy$$ and by \eqref{aggiustamento} we have
\begin{equation}
\label{2-127-nuovo}
\psi^{-1-2\tau}_0(\widetilde{s})\int_{\widetilde{B}^{\varrho}_{\widetilde{s}}}\left\vert v_k\right\vert^2+\psi^{1-2\tau}_0(\widetilde{s})\int_{\widetilde{B}^{\varrho}_{\widetilde{s}}}\left\vert\nabla_{x,y} v_k\right\vert^2
\leq C\left(I_1+I_2+\widetilde{I}_3\right),
\end{equation}
where $C$ depends on $\lambda$, $\Lambda$ and $\Lambda_1$ only. Moreover, by \eqref{2-123}, \eqref{1-125} and \eqref{1-127} we have

\begin{equation}\label{1934}
I_1+I_2+\widetilde{I}_3\leq C \sigma_k(\varepsilon,r_1)
\psi^{-1-2\tau}_0(r_1)+CH^2k^3 5^kC_2^{2k}\psi^{1-2\tau}_0(R).
\end{equation}

Now by \eqref{2-122}, \eqref{2-123}, \eqref{1-125}, \eqref{1-127} and \eqref{1934} we have that, if \eqref{1-123} is satisfied then
\begin{equation}
\label{3-130}
\int_{\widetilde{B}_{\lambda s}} \left\vert v_k\right\vert^2+s^2\int_{\widetilde{B}_{\lambda s}}\left\vert\nabla_{x,y} v_k\right\vert^2
 \leq C\omega_{k,\tau}\quad,
\end{equation}
where $C$ depends on $\lambda$, $\Lambda$ and $\Lambda_1$ only and

\begin{equation}\label{smallomega}
 \omega_{k,\tau}(\varepsilon,r_1)=\sigma_k(\varepsilon,r_1)\left(\frac{\psi_0(\widetilde{s})}{\psi_0(r_1)}\right)^{1+2\tau}
+H^2k^3 5^kC_2^{2k}\left(\frac{\psi_0(\widetilde{s})}{\psi_0(R)}\right)^{1+2\tau}.
\end{equation}

By a standard trace inequality we have

\begin{equation}
\label{1-131}
s\int_{B_{\lambda s/2}} \left\vert v_k(\cdot,0) \right\vert^2 \leq C\omega_{k,\tau}(\varepsilon,r_1)
\end{equation}
and Lemma \eqref{Lemma1} implies

\begin{equation}
\label{1-132}
s\int_{B_{\lambda s/2}} \left\vert u(\cdot,0) \right\vert^2\leq C\left(\frac{\log k}{ \sqrt{k}}+\omega_{k,\tau}(\varepsilon,r_1)\right) ,
\end{equation}
where $C$ depend on $\lambda$, $\Lambda$ and $\Lambda_1$ only.

Now, we choose $k=\tau$ in \eqref{1-132} and using trivial inequality we have that, for any $0<\alpha<\frac{1}{2}$ there exist constants $C_3>1$ and $k_0$ depending on $\lambda$, $\Lambda$, $\Lambda_1$ and $\alpha$ only such that for every $k\geq k_0$ we have
\begin{equation}
\label{1-135}
s\int_{B_{\lambda s/2}} \left\vert u(\cdot,0) \right\vert^2\leq C_3 H_1^2\left[\left(C_3s r_1^{-1}\right)^{ 2k+1}\varepsilon_1^2+\left(C_3s\right)^{2k+1}+k^{-\alpha}\right],
\end{equation}
where
\begin{equation*}
H_1:=H+e\varepsilon \quad \mbox{ and }\quad  \varepsilon_1:=\frac{\varepsilon}{H+e\varepsilon}.
\end{equation*}
Let us denote
\[k_{\ast}= \min\left\{p\in\mathbb{Z}:p \geq \frac{\log \varepsilon_1}{2\log r_1}\right\}.\]
If $k_{\ast}\geq k_0$ then we choose $k=k_{\ast}$ and by \eqref{1-135} we have, for
$$s\leq 1/C_3,$$
\begin{equation}
\label{1-136}
s\int_{B_{\lambda s/2}} \left\vert u(\cdot,0) \right\vert^2\leq 2C_3 H_1^2\left(\varepsilon_1^{2\theta_0}+\left(\frac{2\log (1/r_1)}{\log (1/\varepsilon_1)}\right)^{\alpha}\right),
\end{equation}
where
\begin{equation}
\label{2-136}
\theta_0=\frac{\log (1/C_3s)}{2\log (1/r_1)}.
\end{equation}
Otherwise, if $k_{\ast} < k_0$ then $\frac{\log \varepsilon_1}{2\log r_1}<k_0,$ hence
$$\theta_0 \log (1/\varepsilon_1)=\log (1/C_3s)\frac{\log \varepsilon_1}{2\log r_1}<k_0 \log (1/C_3s).$$
This implies
\[(C_3s)^{-2k_0}\varepsilon_1^{2\theta_0}\geq 1,\]
that, in turns,  taking into account \eqref{limitaz-a-priori}, gives trivially
\begin{gather}
\label{2-137}
\int_{B_{\lambda s/2}} \left\vert u(\cdot,0) \right\vert^2\leq  H^2\leq\\ \nonumber
\leq (C_3s)^{-2k_0}\varepsilon_1^{2\theta_0} H^2\leq (C_3s)^{-2 k_0} (H+e\varepsilon)^{2(1-\theta_0)}\varepsilon^{2\theta_0}.
\end{gather}
Finally, by \eqref{1-136} and \eqref{2-137} we obtain \eqref{SUCP}. $\Box$

\bigskip

\textit{Conclusion of the proof of Theorem \ref{5-115}.}

Let $t_0\in(-T,T)$. It is not restrictive to assume $t_0\geq 0$. Denote
\begin{equation*}
\label{rho-h}
\rho(t_0)=\left(1-T^{-1}t_0\right)\rho_0,   \quad \mbox{ } T(t_0)=\left(1-T^{-1}t_0\right)T
\end{equation*}
and
\begin{equation*}
\label{u-tilde}
U(y,\eta)=u\left(\rho(t_0)y,\eta T(t_0)+t_0\right), \quad \mbox{ for} (y,\eta)\in B_1\times(-1,1).
\end{equation*}
It is easy to check that $\widetilde{u}$ is a solution to

\begin{equation*}
\label{new-equation}
\partial^2_{\eta}U+\widetilde{a}(y)\partial^2_{\eta}U-\mathcal{L}U=0, \quad \mbox{ for} (y,\eta)\in B_1\times(-1,1),
\end{equation*}
where
\begin{subequations}
\label{notation}
\begin{equation*}
\label{notation-a}
\mathcal{L}U=\dive_y\left(\widetilde{A}(y)\nabla_yU\right)+\widetilde{b}(y)\cdot\nabla_yU+\widetilde{c}(y)U,
\end{equation*}
\begin{equation*}
\label{notation-b}
\widetilde{A}(y)=\left(T\rho^{-1}_0\right)^2A\left(\rho(t_0)y\right),\quad \widetilde{a}(y)=(T-t_0)a\left(\rho(t_0)y\right)
\end{equation*}
\begin{equation*}
\label{notation-b}
\widetilde{b}(y)=\left(T(T-t_0)\rho^{-1}_0\right)b\left(\rho(t_0)y\right), \quad \widetilde{c}(y)=\left(T-t_0\right)^2c\left(\rho(t_0)y\right).
\end{equation*}
\end{subequations}
By \eqref{1-65a} and \eqref{2-65} we have respectively
\begin{subequations}
\label{1-65-new}
\begin{equation*}
\label{1-65a-new}
\lambda_0\left\vert\xi\right\vert^2\leq \widetilde{A}(y)\xi\cdot\xi\leq\lambda_0^{-1}\left\vert\xi\right\vert^2, \quad \hbox{for every } x, \xi\in\mathbb{R}^n,
\end{equation*}
\begin{equation*}
\label{2-65-new}
\left\vert \widetilde{A}(y_{\ast})-A(y)\right\vert\leq\frac{\Lambda_0}{\rho_0} \left\vert y_{\ast}-y \right\vert, \quad \hbox{for every } y_{\ast}, y\in\mathbb{R}^n,
\end{equation*}
\end{subequations}
where
$$\lambda_0=\lambda\min\{\left(T\rho^{-1}_0\right)^2,\left(T\rho^{-1}_0\right)^{-2}\} ,\mbox{ and } \Lambda_0=T^2\rho^{-1}_0\Lambda.$$
By \eqref{bound-b} we have
\begin{equation*} \label{bound-b-new}
 \left\vert \widetilde{a}(y)\right\vert+\left\vert\widetilde{b}(y)\right\vert+\left\vert \widetilde{c}(y)\right\vert\leq\Lambda_1, \quad \hbox{for almost every } y\in\mathbb{R}^n.
\end{equation*}
In addition, by \eqref{errore-verticale}, \eqref{limitaz-a-priori} we have respectively
\begin{equation*} \label{errore-verticale-new}
\int_{-1}^{1}\int_{B_{r_0\rho_0^{-1}}}\left\vert U(y,\eta)\right\vert^2dyd\eta\leq \varepsilon^2\left(1-t_0T^{-1}\right)^{-n}
\end{equation*}
and
\begin{equation*} \label{limitaz-a-priori-new}
\max_{\eta\in [-1,1]}\left(\int_{B_1}\left\vert U(y,\eta)\right\vert^2dy+\int_{B_1}\left\vert \partial_{\eta}U(y,\eta)\right\vert^2dy\right)\leq H^2\left(1-t_0T^{-1}\right)^{-n}.
\end{equation*}

Now we apply Proposition \ref{Crucial}. Denoting $s=\rho\rho_0^{-1}$ we have $0<r_0\rho_0^{-1}<s\leq s_0$, therefore

\begin{equation*}
\label{SUCP-new}
\int_{B_s}\left\vert U(y,0) \right\vert^2dy\leq \frac{1}{\left(1-t_0T^{-1}\right)^{n}}\frac{Cs^{-C}(H+e\varepsilon)^2}{\left(\widetilde{\theta}\log \left( \frac{H+e\varepsilon}{\varepsilon}\right) \right)^{\alpha}},
\end{equation*}
where

\begin{equation*}
\label{theta-tilde}
\widetilde{\theta}=\frac{\log (1/Cs)}{\log (1/(r_0\rho_0^{-1}))}.
\end{equation*}

Finally, come back to the variables $x$ and $t$ we get \eqref{SUCP-final}. $\Box$

\bigskip

\textbf{Acknowledgment}

The paper was partially supported by GNAMPA - INdAM.

\end{document}